\documentclass[reqno,12pt]{amsart}
\usepackage{txfonts}
\usepackage{mathrsfs}
\usepackage{amsmath}
\usepackage{amssymb}

\usepackage{amsthm}
\usepackage{verbatim}
\usepackage{latexsym, bm}
\title[On higher dimensional complex Plateau problem]{On higher dimensional complex Plateau problem}

\author[Rong Du]{Rong Du$^{\dag}$}
\address{Department of Mathematics\\
Shanghai Key Laboratory of PMMP\\
East China Normal University\\
Rm. 312, Math. Bldg, No. 500, Dongchuan Road\\
Shanghai, 200241, P. R. China} \email{rdu@math.ecnu.edu.cn}

\author[Yun Gao]{Yun Gao$^{\dag\dag}$}

\address{Department of Mathematics, Shanghai Jiao Tong University,
Shanghai 200240, P. R. of China}

\email{gaoyunmath@sjtu.edu.cn}

\author[Stephen Yau]{Stephen Yau$^{\ast}$}

\address{Department of mathematical sciences\\
Tsinghua University\\
Beijing, 100084, P.R.China} \email{yau@uic.edu,
syau@math.tsinghua.edu.cn}

\thanks{$^{\dag}$The Research Sponsored by the National Natural Science Foundation of China (Grant No. 11471116), Science and Technology Commission of Shanghai Municipality (Grant No. 13dz2260400) and Shanghai Pujiang Program (Grant No. 12PJ1402400).
The first author would also like thank Professor N. Mok for
supporting his research when he was in the University of Hong Kong.}
\thanks{$^{\dag\dag}$ The Research Sponsored by the National Natural Science Foundation of China (Grant No. 11271250,11271251) and SMC program of Shanghai Jiao Tong University.}
\thanks{$^{\ast}$ The Research partially supported by Department of Mathematical Sciences, Tsinghua University, Beijing, China.}
\thanks{All the authors are supported by China NSF (Grant No. 11531007)}

\theoremstyle{definition}
\newtheorem{theorem}[subsection]{Theorem}
\newtheorem{lemma}[subsection]{Lemma}
\newtheorem{definition}[subsection]{Definition}
\newtheorem{proposition}[subsection]{Proposition}

\newtheorem{corollary}[subsection]{Corollary}
\newtheorem{remark}[subsection]{Remark}
\newtheorem{conjecture}[subsection]{Conjecture}
\newfont{\drnew}{wncyr10}
\let\tilde=\widetilde

\allowdisplaybreaks
\def\dashfill{\leaders\hbox{\hbox to 3.25pt{\hrulefill}\hspace*{2pt}\hbox to 3.25pt{\hrulefill}}\hfill}
\newcommand{\CITE}[1]{{[#1]}}
\let\cite=\CITE

\begin{document}

\begin{abstract}
Let $X$ be a compact connected strongly pseudoconvex $CR$ manifold
of real dimension $2n-1$ in $\mathbb{C}^{N}$. It has been an
interesting question to find an intrinsic smoothness criteria for
the complex Plateau problem. For $n\ge 3$ and $N=n+1$, Yau found a
necessary and sufficient condition for the interior regularity of
the Harvey-Lawson solution to the complex Plateau problem by means
of Kohn--Rossi cohomology groups on $X$ in 1981. For $n=2$ and $N\ge
n+1$, the first and third authors introduced a new CR invariant
$g^{(1,1)}(X)$ of $X$. The vanishing of this invariant will give the
interior regularity of the Harvey-Lawson solution up to
normalization. For $n\ge 3$ and $N>n+1$, the problem still remains
open. In this paper, we generalize the invariant $g^{(1,1)}(X)$ to
higher dimension as $g^{(\Lambda^n 1)}(X)$ and show that if
$g^{(\Lambda^n 1)}(X)=0$, then the interior has at most finite
number of rational singularities. In particular, if $X$ is Calabi--Yau of real
dimension $5$, then the vanishing of this invariant is equivalent to
give the interior regularity up to normalization.
\end{abstract}

\maketitle

\vspace{1cm}
\section{\textbf{Introduction}}
One of the natural fundamental questions of complex geometry is to
study the boundaries of complex varieties. For example, the famous
classical complex Plateau problem asks which odd dimensional real
sub-manifolds of $\mathbb{C}^N$ are boundaries of complex
sub-manifolds in $\mathbb{C}^N$. In their beautiful seminal paper,
Harvey and Lawson \cite{Ha-La} proved that for any compact
connected $CR$ manifold $X$ of real dimension $2n-1$, $n\ge 2$, in
$\mathbb{C}^N$, there is a unique complex variety $V$ in
$\mathbb{C}^N$ such that the boundary of $V$ is $X$. In fact, Harvey
and Lawson proved the following theorem.

\vspace{.5cm} \textbf{Theorem} (Harvey-Lawson [Ha-La], [Ha-La2])
Let $X$ be an embeddable strongly pseudoconvex $CR$ manifold. Suppose that $X$ is contained in the boundary of a
strongly pseudoconvex bounded domain in $\mathbb{C}^{N}$. Then
$X$ can be $CR$ embedded in some $\mathbb{C}^{N}$ and $X$
bounds a Stein variety $V \subseteq \mathbb{C}^{N}$ with at
most isolated singularities.

The above theorem is one of the important theorems in complex
geometry. It relates theory of strongly pseudoconvex $CR$ manifolds
on the one hand and the theory of isolated normal singularities on
the other hand.

The next fundamental question is to determine when $X$ is the boundary
of a complex sub-manifold in $\mathbb{C}^N$, i.e., when $V$ is
smooth.  In 1981, Yau [Ya] solved this problem for the case $n\ge 3$
by calculation of Kohn--Rossi cohomology groups $H^{p, q}_{K R}(X)$.
More precisely, suppose $X$ is a compact connected strongly
pseudoconvex $CR$ manifold of real dimension $2n-1$, $n\ge 3$, in
the boundary of a bounded strongly pseudoconvex domain $D$ in
$\mathbb{C}^{n+1}$. Then $X$ is the boundary of a complex
sub-manifold $V\subset D-X$ if and only if Kohn--Rossi cohomology
groups $H^{p, q}_{K R}(X)$ are zeros for $1\le q\le n-2$ (see
Theorem \ref{Yau Pla}).

For $n=2$, i.e. $X$ is a 3-dimensional $CR$ manifold, the intrinsic
smoothness criteria for the complex Plateau problem remains unsolved
for over a quarter of a century even for the hypersurface case. The
main difficulty is that the Kohn--Rossi cohomology groups are
infinite dimensional in this case. Let $V$ be a complex variety with
$X$ as its boundary. Then the singularities of $V$ are surface
singularities. In \cite{Du-Ya}, the first and the third authors
introduced a new $CR$ invariant $g^{(1,1)}(X)$ to solve the
regularity problem of the Harvey-Lawson solution to the complex
Plateau problem. More precisely, they showed that if $X$ is a
strongly pseudoconvex compact Calabi--Yau $CR$ manifold of dimension
$3$ and $X$ is contained in the boundary of a strongly pseudoconvex
bounded domain $D$ in $\mathbb{C}^N$ with holomorphic De Rham
cohomology $H^2_h(X)=0$, then $X$ is the boundary of a complex
sub-manifold up to normalization $V\subset D-X$ with boundary
regularity if and only if $g^{(1,1)}(X)=0$. In particular, if $N=3$,
then $X$ is the boundary of a complex sub-manifold $V\subset D-X$ if
and only if $g^{(1,1)}(X)=0$.

For $n\ge 3$ and $N>n+1$, i.e., non-hypersurface type, the complex
Plateau problem still remains open. In this paper, we generalize the
invariant $g^{(1,1)}(X)$ to higher dimension as $g^{(\Lambda^n
1)}(X)$ and show that if $g^{(\Lambda^n 1)}(X)=0$, then the interior
which $X$ bounds has at most finite number of rational singularities. In
particular, if $X$ is Calabi--Yau of real dimension $5$, i.e.,
$n=3$,  then the vanishing of this invariant is equivalent to give
the interior regularity.

\vspace{.5cm}\textbf{Theorem A} \emph{Let $X$ be a strongly
pseudoconvex compact $CR$ manifold of dimension $2n-1$, where $n>2$.
Suppose that $X$ is contained in the boundary of a strongly
pseudoconvex bounded domain $D$ in $\mathbb{C}^N$. Then $X$ is the
boundary of a variety $V\subset D-X$ with boundary regularity and
the number of non-rational singularities (up to normalization) is
not great than $g^{(\Lambda^n 1)}(X)$. In particular,  if
$g^{(\Lambda^n 1)}(X)=0$, then $V$ has at most finite number of rational singularities.}

\vspace{.5cm}\textbf{Theorem B} \emph{Let $X$ be a strongly
pseudoconvex compact Calabi-Yau $CR$ manifold of dimension $5$.
Suppose that $X$ is contained in the boundary of a strongly
pseudoconvex bounded domain $D$ in $\mathbb{C}^N$. Then $X$ is the
boundary of a complex sub-manifold (up to normalization) $V\subset
D-X$ with boundary regularity if and only if $g^{(\Lambda^3
1)}(X)=0$.}

\vspace{.5cm} In Section 2, we shall recall the definition
Kohn--Rossi cohomology and holomorphic De Rham cohomology for a $CR$
manifold. In Section 3, we survey some known results about the
conjecture of minimal discrepancy and properties of terminal and
rational Gorenstein $3$-folds singularities. In Section 4, we
generalize the invariant of singularities $g^{(1,1)}$ to higher
dimension as $g^{(\Lambda^n 1)}$ and and study some properties of
$g^{(\Lambda^n 1)}$. In Section 5, we use the results in section 4
to solve our main theorems in this paper.

\section{\textbf{Strongly pseudoconvex CR manifolds}}

Kohn--Rossi cohomology was first introduced by Kohn--Rossi. Follow-
ing Tanaka \cite{Ta}, we reformulate the definition in a way
independent of the interior manifold.

\begin{definition}
\emph{Let $X$ be a connected orientable manifold of real dimension
$2n-1$. A $CR$ structure on $X$ is an $(n-1)$-dimensional subbundle
$S$ of $\mathbb{C}T(X)$ (complexified tangent bundle) such that
\begin{enumerate}
\item
$S\bigcap \bar{S}=\{0\}$,
\item
If $L$, $L'$ are local sections of $S$, then so is $[L, L']$.
\end{enumerate}}
\end{definition}

Such a manifold with a $CR$ structure is called a $CR$ manifold.
There is a unique subbundle $\mathcal{H}$ of $T(X)$ such that
$\mathbb{C}\mathcal{H}=S\bigoplus \bar{S}$. Furthermore, there is a
unique homomorphism $J$ : $\mathcal{H}\longrightarrow\mathcal{H}$
such that $J^2=-1$ and $S=\{v-iJv : v\in\mathcal{H}\}$. The pair
$(\mathcal{H}, J)$ is called the real expression of the $CR$
structure.

Let $X$ be a $CR$ manifold with structure $S$. For a complex valued
$C^\infty$ function $u$ defined on $X$, the section
$\bar{\partial}_bu\in \Gamma(\bar{S}^*)$ is defined by
$$\bar{\partial}_bu(\bar{L})=\bar{L}(u), L\in S.$$
The differential operator $\bar{\partial}_b$ is called the
(tangential) Cauchy-Riemann operator, and a solution $u$ of the
equation $\bar{\partial}_bu=0$ is called a holomorphic function.

\begin{definition}\label{hvb on X}
\emph{A complex vector bundle E over X is said to be holomorphic if
there is a differential operator
\[
\bar{\partial}_E: \Gamma(E)\longrightarrow \Gamma(E\otimes
\bar{S}^{*})
\]
satisfying the following conditions:
\begin{enumerate}
\item
$\bar{\partial}_E(fu)(\bar{L}_1)=(\bar{\partial}_bf)(\bar{L}_1)u+f(\bar{\partial}_Eu)(\bar{L}_1)
=(\bar{L}_1f)u+f(\bar{\partial}_Eu)(\bar{L}_1)$,
\item
$(\bar{\partial}_Eu)[\bar{L}_1,
\bar{L}_2]=\bar{\partial}_E(\bar{\partial}_Eu(\bar{L}_2))(\bar{L}_1)-\bar{\partial}_E(\bar{\partial}_Eu(\bar{L}_1))(\bar{L}_2)$,
\end{enumerate}
where $u\in \Gamma(E)$, $f\in C^\infty(X)$ and $L_1$, $L_2 \in
\Gamma(S)$.}
\end{definition}

The operator $\bar{\partial}_E$ is called the Cauchy-Riemann
operator and a solution $u$ of the equation $\bar{\partial}_Eu=0$ is
called a holomorphic cross section.

A basic holomorphic vector bundle over a CR manifold X is the vector
bundle $\widehat{T}(X)=\mathbb{C}T(X)/\bar{S}$. The corresponding
operator $\bar{\partial}=\bar{\partial}_{\widehat{T}(X)}$ is defined
as follows. Let $p$ be the projection from $\mathbb{C}T(X)$ to
$\widehat{T}(X)$. Take any $u \in \Gamma(\widehat{T}(X))$ and
express it as $u=p(Z)$, $Z \in \Gamma(\mathbb{C}T(X))$. For any
$L\in \Gamma(S)$, define a cross section
$(\bar{\partial}u)(\bar{L})$ of $\widehat{T}(X)$ by
$(\bar{\partial}u)(\bar{L})= p([\bar{L},Z])$. One can show that
$(\bar{\partial}u)(\bar{L})$ does not depend on the choice of $Z$
and that $\bar{\partial}u$ gives a cross section of
$\widehat{T}(X)\otimes \bar{S}^{*}$. Furthermore one can show that
the operator $u\longmapsto \bar{\partial}u$ satisfies $(1)$ and
$(2)$ of Definition \ref{hvb on X}, using the Jacobi identity in the
Lie algebra $\Gamma(\mathbb{C}T(X))$. The resulting holomorphic
vector bundle $\widehat{T}(X)$ is called the holomorphic tangent
bundle of $X$.

If $X$ is a real hypersurface in a complex manifold $M$, we may
identify $\widehat{T}(X)$ with the holomorphic vector bundle of all
$(1,0)$ tangent vectors to $M$ and $\widehat{T}(X)$ with the
restriction of $\widehat{T}(M)$ to $X$. In fact, since the structure
$S$ of $X$ is the bundle of all $(1,0)$ tangent vectors to $X$, the
inclusion map $\mathbb{C}T(X)\longrightarrow\mathbb{C}T(M)$ induces
a natural map
$\widehat{T}(X)\xrightarrow[\phantom{ttttt}]{\phi}\widehat{T}(M)|_X$
which is a bundle isomorphism satisfying
$\bar{\partial}(\phi(u))(\bar{L})=\phi(\bar{\partial}u(\bar{L}))$,
$u\in\Gamma(\widehat{T}(X))$, $L\in S$.

For a holomorphic vector bundle $E$ over $X$, set
\[
C^q(X, E)=E\otimes \wedge^q\bar{S}^{*}, \mathscr{C}^q(X,
E)=\Gamma(C^q(X, E))
\]
and define a differential operator
\[
\bar{\partial}^q_E : \mathscr{C}^q(X, E)\longrightarrow
\mathscr{C}^{q+1}(X, E)
\]
by
\[
(\bar{\partial}^q_E\phi)(\bar{L}_1, \dots , \bar{L}_{q+1})
=\sum_i(-1)^{i+1}\bar{\partial}_E(\phi(\bar{L}_1, \dots ,
\widehat{\bar{L}_i}, \dots , \bar{L}_{q+1}))(\bar{L}_i)
\]
\[
+\sum_{i<j}(-1)^{i+j}\phi([\bar{L}_i, \bar{L}_j], \bar{L}_1, \dots ,
\widehat{\bar{L}_i}, \dots , \bar{L}_{q+1})
\]
for all $\phi\in\mathscr{C}^q(X, E)$ and $L_1, \dots ,
L_{q+1}\in\Gamma(S)$. One shows by standard arguments that
$\bar{\partial}^q_E\phi$ gives an element of $\mathscr{C}^{q+1}(X,
E)$ and that $\bar{\partial}^{q+1}_E\bar{\partial}^q_E=0$. The
cohomology groups of the resulting complex $\{\mathscr{C}^{q}(X, E),
\bar{\partial}^{q}_E \}$ is denoted by $H^q(X, E)$.

Let $\{\mathscr{A}^k(X), d\}$ be the De Rham complex of $X$ with
complex coefficients, and let $H^k(X)$ be the De Rham cohomology
groups. There is a natural filtration of the De Rham complex, as
follows. For any integer $p$ and $k$, put
$A^k(X)=\wedge^k(\mathbb{C}T(X)^{*})$ and denote by $F^p(A^k(X))$
the subbundle of $A^k(X)$ consisting of all $\phi\in A^k(X)$ which
satisfy the equality
\[
\phi(Y_1, \dots, Y_{p-1}, \bar{Z}_1, \dots, \bar{Z}_{k-p+1})=0
\]
for all $Y_1, \dots, Y_{p-1} \in \mathbb{C}T(X)_x$ and $Z_1, \dots,
Z_{k-p+1}\in S_x$, $x\in X$. Then
\[
A^k(X)=F^0(A^k(X))\supset F^1(A^k(X)) \supset \cdots
\]
\[
\supset F^k(A^k(X)) \supset F^{k+1}(A^k(X))=0.
\]
Setting $F^p(\mathscr{A}^k(X))=\Gamma(F^p(A^k(X)))$, we have
\[
\mathscr{A}^k(X)=F^0(\mathscr{A}^k(X))\supset
F^1(\mathscr{A}^k(X))\supset \cdots
\]
\[
\supset F^k(\mathscr{A}^k(X))\supset F^{k+1}(\mathscr{A}^k(X))=0.
\]
Since clearly $dF^p(\mathscr{A}^k(X))\subseteq
F^p(\mathscr{A}^{k+1}(X))$, the collection
$\{F^p(\mathscr{A}^k(X))\}$ gives a filtration of the De Rham
complex.

We denote by $H^{p, q}_{KR}(X)$ the groups $E^{p, q}_1(X)$ of the
spectral sequence $\{E^{p, q}_r(X)\}$ associated with the filtration
$\{F^p(\mathscr{A}^k(X))\}$. We call $H^{p, q}_{KR}(X)$ the
Kohn--Rossi cohomology group of type $(p, q)$. More explicitly, let
\[
A^{p, q}(X)=F^p(A^{p+q}(X)), \mathscr{A}^{p, q}(X)=\Gamma(A^{p,
q}(X)),
\]
\[
C^{p, q}(X)=A^{p, q}(X)/A^{p+1, q-1}(X),  \mathscr{C}^{p,
q}(X)=\Gamma(C^{p, q}(X)).
\]
Since $d : \mathscr{A}^{p, q}(X) \longrightarrow \mathscr{A}^{p,
q+1}(X)$ maps $\mathscr{A}^{p+1, q-1}(X)$ into $\mathscr{A}^{p+1,
q}(X)$, it induces an operator $d'': \mathscr{C}^{p,
q}(X)\longrightarrow \mathscr{C}^{p, q+1}(X)$. $H^{p, q}_{KR}(X)$
are then the cohomology groups of the complex $\{\mathscr{C}^{p,
q}(X), d''\}$.

Alternatively $H^{p, q}_{KR}(X)$ may be described in terms of the
vector bundle $E^p=\wedge^p(\widehat{T}(X)^{*})$. If for $\phi\in
\Gamma(E^p)$, $u_1,\dots, u_p\in \Gamma(\widehat{T}(X))$, $Y\in S$,
we define $(\bar{\partial}_{E^p}\phi)(\bar{Y})=\bar{Y}\phi$ by
\[
\bar{Y}\phi(u_1,\dots, u_p)=\bar{Y}(\phi(u_1,\dots,
u_p))+\sum_i(-1)^i\phi(\bar{Y}u_i, u_1,\dots, \widehat{u_i},\dots,
u_p)
\]
where $\bar{Y}u_i=(\bar{\partial}_{\widehat{T}(X)}u_i)(\bar{Y})$,
then we easily verify that $E^p$ with $\bar{\partial}_{E^p}$ is a
holomorphic vector bundle. Tanaka \cite{Ta} proves that $C^{p,q}(X)$
may be identified with $C^q(X, E^p)$ in a natural manner such that
\[
d''\phi=(-1)^p\bar{\partial}_{E^p}\phi, \phi \in \mathscr{C}^{p,
q}(X).
\]
Thus, $H^{p, q}_{KR}(X)$ may be identified with $H^q(X, E^p)$.

We denote by $H^k_h(X)$ the groups $E^{k,0}_2(X)$ of the spectral
sequence $\{E^{p,q}_r(X)\}$ associated with the filtration
$\{F^p(\mathscr{A}^k(X))\}$. We call $H^k_h(X)$ the holomorphic De
Rham cohomology groups. The groups $H^k_h(X)$ are the cohomology
groups of the complex $\{\mathscr{S}^k(X), d\}$, where we put
$\mathscr{S}^k(X)=E^{k,0}_1(X)$ and $d=d_1 :
E^{k,0}_1\longrightarrow E^{k+1,0}_1$. Recall that
$\mathscr{S}^k(X)$ is the kernel of the following mapping:
\[
d_0: E^{k, 0}_0=F^k\mathscr{A}^k=\mathscr{A}^{k,0}(X)\hspace{4cm}
\]
\[
\hspace{3cm}\rightarrow
E^{k,1}_0=F^k\mathscr{A}^{k+1}/F^{k+1}\mathscr{A}^{k+1}=\mathscr{A}^{k,1}(X)/\mathscr{A}^{k+1,0}.
\]
Note that $\mathscr{S}$ may be characterized as the space of
holomorphic $k$-forms, namely holomorphic cross sections of $E^k$.
Thus the complex $\{\mathscr{S}^k(X), d\}$ (respectively, the groups
$H^k_h(X)$ will be called the holomorphic De Rham complex
(respectively, the holomorphic De Rham cohomology groups).

\begin{definition}
\emph{Let $L_1,\dots, L_{n-1}$ be a local frame of the $CR$
structure $S$ on $X$ so that $\bar{L}_1,\dots,\bar{L}_{n-1}$ is a
local frame of $\bar{S}$. Since $S\oplus  \bar{S}$ has complex
codimension one in $\mathbb{C}T(X)$, we may choose a local section N
of $\mathbb{C}T(X)$ such that $L_1,\dots, L_{n-1},
\bar{L}_1,\dots,\bar{L}_{n-1}$, $N$ span $\mathbb{C}T(X)$. We may
assume that $N$ is purely imaginary. Then the matrix $(c_{ij})$
defined by
\[
[L_i,
\bar{L}_j]=\sum_ka^k_{i,j}L_k+\sum_kb^k_{i,j}\bar{L}_k+c_{i,j}N
\]
is Hermitian, and is called the Levi form of $X$.}
\end{definition}
\begin{proposition}
\emph{The number of non-zero eigenvalues and the absolute value of
the signature of $(c_{ij})$ at each point are independent of the
choice of $L_1,\dots, L_{n-1}, N$.}
\end{proposition}
\begin{definition}
\emph{$X$ is said to be strongly pseudoconvex if the Levi form is
positive definite at each point of $X$.}
\end{definition}

\begin{definition}
\emph{Let $X$ be a CR manifold of real dimension $2n-1$. $X$ is said
to be Calabi-Yau if there exists a nowhere vanishing holomorphic
section in $\Gamma(\wedge^n\widehat{T}(X)^*)$, where
$\widehat{T}(X)$ is the holomorphic tangent bundle of $X$.}
\end{definition}

\textbf{Remark}:
\begin{enumerate}
\item
Let $X$ be a $CR$ manifold of real dimension $2n-1$ in
$\mathbb{C}^n$. Then $X$ is a Calabi-Yau $CR$ manifold.
\item
Let $X$ be a strongly pseudoconvex $CR$ manifold of real dimension
$2n-1$ contained in the boundary of bounded strongly pseudoconvex
domain in $\mathbb{C}^{n+1}$. Then $X $ is a Calabi-Yau $CR$
manifold.
\end{enumerate}
The proof of these two statements is essentially the fact that any hypersurface singularities are Gorenstein and with the same arguments as Lemma \ref{boundary} to get  a nowhere vanishing holomorphic
section of the holomorphic tangent bundle of $X$.

\section{\textbf{Minimal discrepancy and 3-dimensional canonical Gorenstein singularities}}
Canonical singularities appear as singularities of the canonical
model of a projective variety, and terminal singularities are
special cases that appear as singularities of minimal models. They
were introduced by Reid in 1980 (\cite{Re}). Terminal singularities
are important in the minimal model program because smooth minimal
models do not always exist, and thus one must allow certain
singularities, namely the terminal singularities.

Suppose that $X$ is a normal variety such that its canonical class
$K_X$ is $\mathbb{Q}$-Cartier, and let $f:Y\rightarrow X$ be a
resolution of the singularities of $X$. Then
\[K_Y=f^*K_X+\sum_ia_iE_i,\]
where the sum is over the irreducible exceptional divisors, and the
$a_i$ are rational numbers, called the discrepancies.

Then the singularities of $X$ are called:
\begin{equation} \label{defsing}
\left\{ \begin{split}
         &\text{terminal}& a_i>0 &&\text{for all i}\\
         &\text{canonical}& a_i\ge 0 &&\text{for all i} \\
         &\text{log-terminal} & a_i > -1 &&\text{for all i} \\
         &\text{log-canonical} & a_i \ge -1 &&\text{for all i.}
                          \end{split} \right.
                          \end{equation}
\begin{definition}
The minimal discrepancy of a variety $X$ at $0$, denoted by
$Md_0(X)$ (or $Md(X)$ for short), is the minimum of all
discrepancies of discrete valuations of $\mathbb{C}(X)$, whose
center on $X$ is $0$.
\end{definition}
\begin{remark}
The minimal discrepancy only exists when $X$ has log-canonical
singularities (see, e.g. \cite{C-K-M}). Whenever $Md(X)$ exists it
is at least $-1$.
\end{remark}

Shokurov conjecture that the minimal discrepancy is bounded above in
term of the dimension of a variety.

\begin{conjecture}\label{co} (Shokurov \cite{Sh}): The minimal
discrepancy $Md_0(X)$ of a variety $X$ at $0$ of dimension $n$ is at
most $n-1$. Moreover, if $Md_0(X)=n-1$, then $(X, 0)$ is
nonsingular.
\end{conjecture}

The conjecture was confirmed for surfaces (\cite{Al}) and
$3$-dimensional singularities after the explicit classification
(\cite{Re2}) of Gorenstein terminal $3$-fold singularities with
\cite{E-M-Y} or \cite{Ma}. If $X$ is a local complete intersection,
then the conjecture also holds (see \cite{E-M} and \cite{E-M-Y}).

In this paper, we are going to consider the $3$-dimensional
singularities. Mori (\cite{Mo}), Cutkosky (\cite{Cu}) and Brenton
(\cite{Br}) had some classification theorems about special $3$-folds
singularities. The following two theorems will be used to prove our
main theorems.
\begin{theorem}\label{ratGor}(\cite{Re} Corollary 2.14, Corollary 2.12)
Let $0\in V$ be a $3$-dimensional rational Gorenstein singularity,
then there exists a partial resolution $\pi:\tilde{V}\rightarrow V$
such that $\pi^{-1}\{0\}$ is a union of nonsingular rational or
elliptic ruled surfaces, $\tilde{V}$ only has terminal singularities
and $K_{\tilde{V}}=\pi^*K_V$.
\end{theorem}

\begin{remark}
The information of the irreducible and reduced components of the
exceptional set can also be found in \cite{Br}.
\end{remark}

\begin{definition}
A singular point $0$ is called compound Du Val ($cDV$) if for a
general section $H$ through $0$, $0\in H$ is a Du Val singularities.
\end{definition}

\begin{remark}
A cDV singularity is formally equivalent to the germ of a
hypersurface singularity $(\{f=0\}, 0)$ in $\mathbb{C}^4$, where
\begin{equation}
f(x_0,x_1,x_2,x_3)=f_{X_n}(x_0,x_1,x_2)+x_3g(x_0,x_1,x_2,x_3),
\end{equation}
where $X_n$ stands for $A_n$, $D_n$ or $E_n$, $g(0,0,0,0)=0$ and
$f_{X_n}$ is one of the following polynomials:
\begin{align}
f_{A_n}&=x_0^2+x_1^2+x_2^{n+1} \quad (n\ge 1)\\
f_{D_n}&=x_0^2+x_1^2x_2+x_2^{n-1} \quad  (n\ge 4)\\
f_{E_6}&=x_0^2+x_1^3+x_2^{4} \\
f_{E_7}&=x_0^2+x_1^3+x_1x_2^{3}\\
f_{E_8}&=x_0^2+x_1^3+x_2^{5}.
\end{align}

\end{remark}

\begin{theorem}\label{cDV}(\cite{Re2} Theorem 1.1)
Let $0\in V$ be a $3$-dimensional singularity. Then $0$ is an
isolated $cDV$ singularity if and only if $0$ is Gorenstein
terminal.
\end{theorem}

%\begin{corollary}\label{nGt}
%Let $0\in V$ be a $3$-dimensional terminal Gorenstein singularity,
%then there exists a resolution $\pi:\tilde{V}\rightarrow V$ such
%that one irreducible and reduced component of the exceptional set
%$F$ is a nonsingular rational surfaces and the discrepancy of $F$ is
%1.
%\end{corollary}
%\begin{proof}
%Form Theorem \ref{cDV}, $0$ is a $cDV$ singularity. So $0$ is an
%isolated hypersurface singularity. Take a blowing-up at $0$, the
%exceptional set, i.e., the projectivised tangent cone denoted by
%$F'$ is a subscheme of degree $2$ in $\mathbb{P}^3$ whose
%irreducible and reduced nonsingular model is a rational surface
%denoted by $F'$.
%\end{proof}

%\begin{remark}
%\cite{Cu} or \cite{Mo} gives a precisely description of the rational
%surfaces mentioned above.
%\end{remark}

\section{\textbf{New Invariants of singularities and new $CR$-invariants}}
Let V be an $n$-dimensional complex analytic subvariety in
$\mathbb{C}^N$ with only isolated singularities. In [Ya2], Yau
considered four kinds of sheaves of germs of holomorphic $p$-forms
\begin{enumerate}
\item
$\bar{\Omega}^p_V:=\pi_*\Omega^p_M$, where $\pi: M\longrightarrow V$
is a resolution of singularities of $V$.
\item
$\bar{\bar{\Omega}}^p_V:=\theta_*\Omega^p_{V\backslash V_{sing}}$
where $\theta : V\backslash V_{sing}\longrightarrow V$ is the
inclusion map and $V_{sing}$ is the singular set of $V$.
\item
$\Omega^p_V:=\Omega_{\mathbb{C}^N}^p/\mathscr{K}^p$, where
$\mathscr{K}^p=\{f\alpha+dg\wedge\beta :
\alpha\in\Omega_{\mathbb{C}^N}^p; \beta\in
\Omega_{\mathbb{C}^N}^{p-1}; f, g\in\mathscr{I}\}$ and $\mathscr{I}$
is the ideal sheaf of $V$ in $\mathbb{C}^N$.
\item
$\widetilde{\Omega}^p_V:=\Omega_{\mathbb{C}^N}^p/\mathscr{H}^p$,
where $\mathscr{H}^p=\{\omega\in\Omega_{\mathbb{C}^N}^p:
\omega|_{V\backslash V_{sing}}=0\}$.
\end{enumerate}

Clearly ${\Omega}^p_V$, $\widetilde{\Omega}^p_V$ are coherent.
$\bar{\Omega}^p_V$ is a coherent sheaf because $\pi$ is a proper
map. $\bar{\bar{\Omega}}^p_V$ is also a coherent sheaf by a theorem
of Siu (cf Theorem A of \cite{Si}). In case $V$ is a normal variety,
the dualizing sheaf $\omega_V$ of Grothendieck is actually the sheaf
$\bar{\bar{\Omega}}^n_V$.

\begin{definition}\label{s inv}
\emph{Let $V$ be an $n$-dimensional Stein space with $0$ as its only
singular point. Let $\pi: (M, A)\rightarrow (V, 0)$ be a resolution
of the singularity with $A$ as exceptional set. The geometric genus
$p_g$ and the irregularity $q$ of the singularity are defined as
follows (cf. \cite{Ya2}, \cite{St-St}):
\begin{equation}
p_g:= \dim \Gamma(M\backslash A, \Omega^n)/\Gamma(M, \Omega^n),
\end{equation}
\begin{equation}
q:= \dim \Gamma(M\backslash A, \Omega^{n-1})/\Gamma(M, \Omega^{n-1}),
\end{equation}
\begin{equation}
g^{(p)}:= \dim \Gamma(M, \Omega^{p}_M)/\pi^{*}\Gamma(V, \Omega^p_V).
\end{equation}}
\end{definition}

Let $X$ be a compact connected strongly pseudoconvex $CR$ manifold
of real dimension $2n-1$, in the boundary of a bounded strongly
pseudoconvex domain $D$ in $\mathbb{C}^N$. By a result of Harvey and
Lawson, there is a unique complex variety $V$ in $\mathbb{C}^N$ such
that the boundary of $V$ is $X$. Let $\pi: (M, A_1,\cdots,
A_k)\rightarrow (V, 0_1,\cdots, 0_k)$ be a resolution of the
singularities with $A_i=\pi^{-1}(0_i)$, $1\le i\le k$, as
exceptional sets.

In order to solve the classical complex Plateau problem, we need to
find some $CR$-invariant which can be calculated directly from the
boundary $X$ and the vanishing of this invariant will give
regularity of Harvey-Lawson solution to complex Plateau problem.

For this purpose, we define a new sheaf $\bar{\bar{\Omega}}_V^{1,
1}$, new invariant of surface singularities $g^{(1,1)}$ and new $CR$
invariant $g^{(1,1)}(X)$ in \cite{Du-Ya}.  Now, we are going to
generalize them to higher dimension for dealing with general complex
Plateau problem.

\begin{definition}
\emph{Let $(V, 0)$ be a Stein germ of an $n$-dimensional analytic
space with an isolated singularity at $0$. Define a sheaf of germs
$\bar{\bar{\Omega}}_V^{\Lambda^p 1}$ by the sheaf associated with the
presheaf
\[
U\mapsto \langle\Lambda^p\Gamma(U, \bar{\bar{\Omega}}^1_{V})\rangle,
\]
where $U$ is an open set of $V$ and $2 \le p\le n$.}
\end{definition}

\begin{lemma}\label{loc inv2}
\emph{Let $V$ be an $n$-dimensional Stein space with $0$ as its only
singular point in $\mathbb{C}^N$. Let $\pi: (M, A)\rightarrow (V,
0)$ be a resolution of the singularity with $A$ as exceptional set.
Then $\bar{\bar{\Omega}}_V^{\Lambda^p 1}$ is coherent and there is a
short exact sequence
\begin{equation}
0\longrightarrow\bar{\bar{\Omega}}_V^{\Lambda^p
1}\longrightarrow\bar{\bar{\Omega}}_V^p\longrightarrow\mathscr{G}^{(\Lambda^p
1)}\longrightarrow 0
\end{equation}
where $\mathscr{G}^{(\Lambda^p 1)}$ is a sheaf supported on the
singular point of $V$. Let
\begin{equation}
G^{(\Lambda^p 1)}(M\backslash A):=\Gamma(M\backslash A,
\Omega^p_M)/\langle\Lambda^p \Gamma(M\backslash A, \Omega^1_M)\rangle,
\end{equation}
then $\dim \mathscr{G}^{(\Lambda^p 1)}_0=\dim G^{(\Lambda^p
1)}(M\backslash A)$.}
\end{lemma}
\begin{proof}
Since the sheaf of germ $\bar{\bar{\Omega}}^p_V$ is coherent by a
theorem of Siu (cf Theorem A of \cite{Si}), for any point $w\in V$
there exists an open neighborhood $U$ of $w$ in $V$ such that
$\Gamma(U, \bar{\bar{\Omega}}^1_{V})$ is finitely generated over
$\Gamma(U, \mathscr{O}_{V})$. So $\Gamma(U,
\Lambda^p\bar{\bar{\Omega}}^1_{V}))$ is finitely generated over
$\Gamma(U, \mathscr{O}_{V})$, which means $\Gamma(U,
\bar{\bar{\Omega}}_{V}^{\Lambda^p 1})$ is finitely generated over
$\Gamma(U, \mathscr{O}_{V})$, i.e.
$\bar{\bar{\Omega}}_{V}^{\Lambda^p 1}$ is a sheaf of finite type. It
is obvious that $\bar{\bar{\Omega}}_{V}^{\Lambda^p 1}$ is a subsheaf
of $\bar{\bar{\Omega}}_{V}^p$ which is also coherent. So
$\bar{\bar{\Omega}}_{V}^{\Lambda^p 1}$ is coherent.

Notice that the stalk of $\bar{\bar{\Omega}}_V^{\Lambda^p 1}$ and
$\bar{\bar{\Omega}}_V^p$ coincide at each point different from the
singular point $0$, $\mathscr{G}^{(\Lambda^p 1)}$ is supported at
$0$. And from Cartan Theorem B
$$\dim
\mathscr{G}^{(\Lambda^p 1)}_0=\dim \Gamma(V,
\bar{\bar{\Omega}}_V^p)/\Gamma(V, \bar{\bar{\Omega}}_V^{\Lambda^p
1}) =\dim G^{(\Lambda^p 1)}(M\backslash A).$$
\end{proof}

Thus, from Lemma \ref{loc inv2}, we can define a local invariant of
a singularity which is independent of resolution.
\begin{definition}
\emph{Let $V$ be an $n$-dimensional Stein space with $0$ as its only
singular point. Let $\pi: (M, A)\rightarrow (V, 0)$ be a resolution
of the singularity with $A$ as exceptional set. Let
\begin{equation}
g^{(\Lambda^p 1)}(0):=\dim \mathscr{G}^{(\Lambda^p 1)}_0=\dim
G^{(\Lambda^p 1)}(M\backslash A).
\end{equation}}
\end{definition}

We will omit $0$ in $g^{(\Lambda^p 1)}(0)$ if there is no confusion
from the context.

Let $\pi: (M, A_1,\cdots, A_k)\rightarrow (V, 0_1,\cdots, 0_k)$ be a
resolution of the singularities with $A_i=\pi^{-1}(0_i)$, $1\le i\le
k$, as exceptional sets. In this case we still let
$$G^{(\Lambda^p
1)}(M\backslash A):=\Gamma(M\backslash A, \Omega^p_M)/\langle\Lambda^p
\Gamma(M\backslash A, \Omega^1_M)\rangle,$$ where $A=\cup_i A_i$.

\begin{definition}
\emph{If $X$ is a compact connected strongly pseudoconvex $CR$
manifold of real dimension $2n-1$, in the boundary of a bounded
strongly pseudoconvex domain $D$ in $\mathbb{C}^N$. Suppose $V$ in
$\mathbb{C}^N$ such that the boundary of $V$ is $X$. Let $\pi: (M,
A=\cup_i A_i)\rightarrow (V, 0_1,\cdots, 0_k)$ be a resolution of
the singularities with $A_i=\pi^{-1}(0_i)$, $1\le i\le k$, as
exceptional sets. Let
\begin{equation}
G^{(\Lambda^p 1)}(M\backslash A):=\Gamma(M\backslash A ,
\Omega^p_M)/\langle\Lambda^p \Gamma(M\backslash A, \Omega^1_M)\rangle
\end{equation}
and
\begin{equation}
G^{(\Lambda^p 1)}(X):=\mathscr{S}^p(X)/\langle\Lambda^p \mathscr{S}^1(X)\rangle,
\end{equation}
where $\mathscr{S}^q$ are holomorphic cross sections of
$\wedge^q(\widehat{T}(X)^*)$. Then we set
\begin{equation}
g^{(\Lambda^p 1)}(M\backslash A):=\dim G^{(\Lambda^p 1)}(M\backslash
A),
\end{equation}
\begin{equation}
g^{(\Lambda^p 1)}(X):=\dim G^{(\Lambda^p 1)}(X).
\end{equation}}
\end{definition}

\begin{lemma}\label{boundary}
\emph{Let $X$ be a compact connected strongly pseudoconvex $CR$
manifold of real dimension $2n-1$ which bounds a bounded strongly
pseudoconvex variety $V$ with only isolated singularities
$\{0_1,\cdots, 0_k\}$ in $\mathbb{C}^N$. Let $\pi: (M, A_1,\cdots,
A_k)\rightarrow (V, 0_1,\cdots, 0_k)$ be a resolution of the
singularities with $A_i=\pi^{-1}(0_i)$, $1\le i\le k$, as
exceptional sets. Then $g^{(\Lambda^p 1)}(X)=g^{(\Lambda^p
1)}(M\backslash A)$, where $A=\cup A_i$, $1\le i\le k$.}
\end{lemma}
\begin{proof}
Take a one-convex exhausting function $\phi$ on $M$ such that
$\phi\ge 0$ on $M$ and $\phi(y)=0$ if and only if $y\in A$. Set
$M_r=\{y\in M, \phi(y)\ge r\}$. Since $X=\partial M$ is strictly
pseudoconvex, any holomorphic $q$-form $\theta\in \mathscr{S}^q(X)$
can be extended to a one side neighborhood of $X$ in $M$. Hence
$\theta$ can be thought of as holomorphic $q$-form on $M_r$, i.e. an
element in $\Gamma(M_r, \Omega^q_{M_r})$. By Andreotti and Grauert
(\cite{An-Gr}), $\Gamma(M_r, \Omega^q_{M_r})$ is isomorphic to
$\Gamma(M\backslash A, \Omega^q_{M})$. So $g^{(\Lambda^p
1)}(X)=g^{(\Lambda^p 1)}(M\backslash A)$.
\end{proof}

 By Lemma \ref{boundary} and the proof of Lemma \ref{loc inv2}, we can get the following lemma easily.

\begin{lemma}\label{gp1}
\emph{Let $X$ be a compact connected strongly pseudoconvex $CR$
manifold of real dimension $2n-1$, which bounds a bounded strongly
pseudoconvex variety $V$ with only isolated singularities
$\{0_1,\cdots, 0_k\}$ in $\mathbb{C}^N$. Then $g^{(\Lambda^p
1)}(X)=\sum_i g^{(\Lambda^p 1)}(0_i)=\sum_i \dim
\mathscr{G}^{(\Lambda^p 1)}_{0_i}$.}
\end{lemma}

The following proposition is to show that $g^{(\Lambda^p 1)}$ is
bounded above.
\begin{proposition}
\emph{Let $V$ be an $n$-dimensional Stein space with $0$ as its only
singular point. Then
\begin{equation*}
g^{(\Lambda^p 1)}\le \left\{ \begin{aligned}
         g^{(p)},&\quad p\le n-2; \\
                  g^{(n-1)}+q,&\quad p=n-1;\\
                  g^{(n)}+p_g,&\quad p=n.
                          \end{aligned} \right.
                          \end{equation*}
}
\end{proposition}

\begin{proof}
Since $$g^{(\Lambda^p 1)}=\dim\Gamma(M\backslash A,
\Omega^p_M)/\langle\Lambda^p \Gamma(M\backslash A, \Omega^1_M)\rangle,$$
\[
g^{(p)}= \dim \Gamma(M, \Omega^{p})/\pi^{*}\Gamma(V, \Omega^p_V),
\]
\begin{equation*}
\dim \Gamma(M\backslash A, \Omega^p_M)/\Gamma(M, \Omega^p_M)= \left\{
\begin{aligned}
         0,&\quad p\le n-2; \\
                  q,&\quad p=n-1;\\
                  p_g,&\quad p=n.
                          \end{aligned} \right.
                          \end{equation*}
and
\begin{equation}
\begin{split}
\pi^{*}\Gamma(V, \Omega^p_V)&=\langle\pi^{*}(\Lambda^p \Gamma(V,
\Omega^1_V)\rangle\\
&\subseteq \Lambda^p\Gamma(M, \Omega^1_M)\\
&\subseteq \Lambda^p\Gamma(M\backslash A, \Omega^1_M),
\end{split}
 \end{equation} the result follows easily.
\end{proof}

The following theorem is the crucial part for solving the classical
complex Plateau problem of real dimension $3$.

\begin{theorem}(\cite{Du-Ya})\label{new inv}
\emph{Let $V$ be a $2$-dimensional Stein space with $0$ as its only
normal singular point with $\mathbb{C}^*$-action. Let $\pi: (M,
A)\rightarrow (V, 0)$ be a minimal good resolution of the
singularity with $A$ as exceptional set, then $g^{(\Lambda^2 1)}\ge
1$.}
\end{theorem}

\begin{remark}
We also show that $g^{(\Lambda^2 1)}$ is strictly positive for
rational singularities (\cite{Du-Ga}) and minimal elliptic
singularities (\cite{Du-Ga2}) and exact $1$ for rational double
points, triple points and quotient singularities (\cite{Du-Lu-Ya}).
\end{remark}
Similarly, the following theorem is the crucial part for solving the
classical complex Plateau problem of real dimension $5$.
\begin{theorem}\label{newn}
\emph{Let $V$ be an $n$-dimensional Stein space with $0$ as its only
non-rational singular point, where $n>2$, then $g^{(\Lambda^n 1)}\ge
1$.}
\end{theorem}
\begin{proof}
Suppose $\pi: M\rightarrow V$ be any resolution of the singularity
$0$ with $E$ as its exceptional set. By a result of Greuel
(\cite{Gr}, Proposition 2.3), for every holomorphic $1$ form $\eta$
on $V-0$, $\pi^*(\eta)$ can extends holomorphically to $M$.
Since $0$ is not rational, there exists a holomorphic $n$ form
$\omega$ on $V-0$ such that $\pi^*(\omega)$ can not extends
holomorphically to $M$. So $$g^{(\Lambda^n 1)}\ge
\dim\Gamma(M-E, \Omega_M^n)/\Lambda^n \Gamma(M,
\Omega_M^1)>0.$$
\end{proof}

\begin{theorem}\label{new3}
\emph{Let $V$ be a $3$-dimensional Stein space with $0$ as its only
normal Gorenstein singular point, then $g^{(\Lambda^3 1)}\ge 1$.}
\end{theorem}
\begin{proof}
If $0$ is non-rational, then $g^{(\Lambda^3 1)}\ge 1$ by Theorem
\ref{newn}. So we only need to show that the result is true for
rational Gorenstein singularities. It is well known that rational
Gorenstein singularities are canonical (see \cite{Ko-Mo} Corollary
5.24).
%Let $\pi: M\rightarrow V$ be
%a resolution with $E=\cup E_i$ as its exceptional set, where each
%$E_i$ is the non-singular irreducible component. We can assume
%without loss of generality that the exceptional set $E$ is a divisor
%with normal crossings.
We are going to separate our argument into
two cases to finish the proof.
\begin{itemize}
\item [Case i]: If $0$ is terminal, by Theorem \ref{cDV}, $0$ is a cDV
singularity defined by $f(x_0,x_1,x_2,x_3)=0$. Take a typical
blowing-up $\sigma: V'\rightarrow V$ at $0$, the exceptional set,
i.e., the projectivised tangent cone is a subscheme of degree $2$ in
$\mathbb{P}^3$ whose irreducible and reduced component denoted by
$F$ is a nonsingular rational surface after desingularization.
Consider
\[s=\frac{dx_1\wedge dx_2\wedge dx_3}{\partial f/\partial x_0}.\]
A typical piece of the blowing-up of $\mathbb{C}^4$ at $0$ has
coordinates $y_0$, $y_1$, $y_2$ and $y_3$ with $x_0=y_0, x_1=y_0y_1,
x_2=y_0y_2, x_3=y_0y_3$, and in this piece the nonsingular proper
transform is given by $f(y_0, y_0y_1, y_0y_2, y_0y_3)/y_0^2$. Then
\[dx_i=y_idy_0+y_0dy_i, \quad i=1,2,3\]
and the vanishing order of $\partial f(y_0, y_0y_1, y_0y_2,
y_0y_3)/\partial y_0$ along $y_0$ is $1$. Then
\begin{equation}
\begin{split}
\sigma^*s&=\frac{\bigwedge_{i=1}^3(y_idy_0+y_0dy_i)}{\partial
f(y_0, y_0y_1, y_0y_2, y_0y_3)/\partial y_0}\\
 &=\frac{y_0^2\Theta(y_0,y_1,y_2,y_3)}{\partial
f(y_0, y_0y_1, y_0y_2, y_0y_3)/\partial y_0},
 \end{split}
 \end{equation}
 where
\begin{equation}
\begin{split}
 &\Theta(y_0,y_1,y_2,y_3)\\
 =&y_1dy_0\wedge dy_2 \wedge dy_3+y_2 dy_1\wedge dy_0 \wedge dy_3+\\
 &y_3 dy_1\wedge dy_2 \wedge dy_0+y_0 dy_1\wedge dy_2 \wedge dy_3.
 \end{split}
 \end{equation}
So the vanishing order of $\sigma^*s$ along $F$ is $1$.

Let $\pi: M\rightarrow V'$ be a resolution consists of a series of
blowing-ups with $E=\cup E_i$ as the exceptional set of $\pi \circ
\sigma$, where each $E_i$ is the non-singular irreducible component.
We can assume without loss of generality that the exceptional set
$E$ is a divisor with normal crossings. So $(\pi \circ\sigma)^*s\in
\Gamma(M, \Omega_M^3)$ vanishes along some $E_j=\pi^*F\subseteq E$
of order $1$, which is a nonsingular rational surface, i.e.
Ord$_{E_j} (\pi \circ\sigma)^*s=1$. Take a tubular neighborhood
$M_j$ of $E_j$ such that $M_j\subset M$. Consider the exact sequence
(\cite{E-V})
\begin{equation}\label{logEj}
0\rightarrow \Omega_{M_j}^1(\log E_j)(-E_j)\rightarrow
\Omega_{M_j}^1\rightarrow \Omega_{E_j}^{1}\rightarrow 0.
\end{equation}
By taking global sections we have
\begin{equation}\label{loggEj}
0\rightarrow \Gamma({M_j}, \Omega_{M_j}^1(\log E_j)(-E_j))\rightarrow
\Gamma({M_j}, \Omega_{M_j}^1)\rightarrow \Gamma(E_j,
\Omega_{E_j}^{1}).
\end{equation}
Since $E_j$ is a nonsingular rational surface, $h^1(E_j,
\mathcal{O}_{E_j})$, the irregularity of $E_j$, is $0$. Then
$\Gamma(E_j, \Omega_{E_j}^{1})=0$ by Hodge symmetry. Therefore
\begin{equation}\label{equal}
\Gamma(M_j, \Omega_{M_j}^1(\log E_j)(-E_j))=\Gamma(M_j,
\Omega_{M_j}^1)
\end{equation}
from (\ref{loggEj}).

Suppose $\eta\in \Gamma(M_j, \Omega_{M_j}^1)$, then $\eta\in
\Gamma(M_j, \Omega_{M_j}^1(\log E_j)(-E_j))$ by (\ref{equal}). Chose
a point $P$ in $E_j$ which is a smooth point in $E$. Let $(x_1, x_2,
x_3)$ be a coordinate system center at $P$ such that $E_j$ is given
locally by $x_1=0$. Write $\eta$ locally around $P$ : $\eta\circeq
f_1dx_1+f_2x_1dx_2+f_3x_1dx_3$, where $f_1$, $f_2$ and $f_3$ are
holomorphic functions and ``$\circeq$" means local equality around
$P$. So the vanishing order of any elements in $\Lambda^3\Gamma(M,
\Omega_M^1)$ along the irreducible exceptional set $E_j$ is at least
$2$ by noticing $\Gamma(M, \Omega_M^1)\subseteq \Gamma(M_j,
\Omega_{M_j}^1)$ under natural restriction. So
\[g^{(\Lambda^3 1)}=\dim\Gamma(M, \Omega_M^3)/\Lambda^3 \Gamma(M,
\Omega_M^1)\ge 1.\]

\item [Case ii]: If $0$ is canonical but not terminal, then by
Theorem \ref{ratGor}, there exists a partial resolution
$\rho:\tilde{V}\rightarrow V$ such that $\rho^{-1}\{0\}=\cup_i F_i$
is a union of nonsingular rational or elliptic ruled surfaces,
$\tilde{V}$ only has terminal singularities and
$K_{\tilde{V}}=\rho^*K_V$. So if we let $\pi: M\rightarrow
\tilde{V}$ be a resolution consists a series of blowing-ups, then
the discrepancy of some $\pi^*F_j$ for $V$ is $0$. Therefore there
is a section $s\in \Gamma(M, \Omega_M^3)$ such that $s$ does not
vanish along some irreducible exceptional set $E_j:=\pi^*F_j$, i.e.
Ord$_{E_j} s=0$. Take a tubular neighborhood $M_j$ of $E_j$ such
that $M_j\subset M$. Consider the same exact sequence as in Case i:
\begin{equation}
0\rightarrow \Omega_{M_j}^1(\log E_j)(-E_j)\rightarrow
\Omega_{M_j}^1\rightarrow \Omega_{E_j}^{1}\rightarrow 0.
\end{equation}
By taking global sections we have
\begin{equation}\label{rep}
0\rightarrow \Gamma({M_j}, \Omega_{M_j}^1(\log E_j)(-E_j))\rightarrow
\Gamma({M_j}, \Omega_{M_j}^1)\rightarrow \Gamma(E_j,
\Omega_{E_j}^{1}).
\end{equation}
We know that $E_j$ must be a rational surface or elliptic ruled
surface. If $E_j$ is rational, we have $h^1(E_j,
\mathcal{O}_{E_j})$, the irregularity of $E_j$, is $0$. Then
$\Gamma(E_j, \Omega_{E_j}^{1})=0$ by Hodge symmetry. Therefore
\begin{equation}\label{equalEj}
\Gamma(M_j, \Omega_{M_j}^1(\log E_j)(-E_j))=\Gamma(M_j,
\Omega_{M_j}^1)
\end{equation}
from (\ref{rep}). Then by the same local argument as in Case i and
$\Gamma(M, \Omega_M^1)\subseteq \Gamma(M_j, \Omega_{M_j}^1)$, we
have
\[g^{(\Lambda^3 1)}=\dim\Gamma(M,
\Omega_M^3)/\Lambda^3 \Gamma(M, \Omega_M^1)\ge 1.\]

If $E_j$ is elliptic ruled surface, the only difference is that
$h^1(E_j, \mathcal{O}_{E_j})$, the irregularity of $E_j$, is $1$.
Then $\Gamma(E_j, \Omega_{E_j}^{1})=1$ by Hodge symmetry. Therefore
\begin{equation}\label{1dim}
\dim\Gamma(M_j, \Omega_{M_j}^1)/\Gamma(M_j, \Omega_{M_j}^1(\log
E_j)(-E_j))\le 1
\end{equation}
from (\ref{loggEj}).

If we take three $\mathbb{C}$-linear independent holomorphic
$1$-forms in $\Gamma(M, \Omega_M^1)$, then there must exist two
elements $\eta_1, \eta_2 \in \Gamma(M, \Omega_M^1)$ such that
$\eta_1|_{M_j}, \eta_2|_{M_j}\in \Gamma(M_j, \Omega_{M_j}^1(\log
E_j)(-E_j))$ from (\ref{1dim}). Similarly, chose a point $P$ in
$E_j$ which is a smooth point in $E$. Let $(x_1, x_2, x_3)$ be a
coordinate system center at $P$ such that $E_j$ is given locally by
$x_1=0$. Write $\eta_1$ and $\eta_2$ locally around $P$ :
$$\eta_1\circeq f_1dx_1+f_2x_1dx_2+f_3x_1dx_3$$ and
\[\eta_2\circeq
g_1dx_1+g_2x_1dx_2+g_3x_1dx_3,\] where $f_i$ and $g_i$ ($1 \le i\le
3$) are holomorphic functions and ``$\circeq$" means local equality
around $P$. So the vanishing order of any elements in
$\Lambda^3\Gamma(M_j, \Omega_{M_j}^1)$ along the irreducible
exceptional set $E_j$ is at least $1$. So by noticing $\Gamma(M,
\Omega_M^1)\subseteq \Gamma(M_j, \Omega_{M_j}^1)$, we have
\[g^{(\Lambda^3 1)}=\dim\Gamma(M, \Omega_M^3)/\Lambda^3 \Gamma(M,
\Omega_M^1)\ge 1.\]
\end{itemize}
\end{proof}

\section{\textbf{The classical complex Plateau problem}}
In 1981, Yau [Ya] solved the classical complex Plateau problem for
the case $n\ge 3$.
\begin{theorem}([Ya])\label{Yau Pla}
\emph{Let $X$ be a compact connected strongly pseudoconvex $CR$
manifold of real dimension $2n-1$, $n\ge 3$, in the boundary of a
bounded strongly pseudoconvex domain $D$ in $\mathbb{C}^{n+1}$. Then
$X$ is the boundary of a complex sub-manifold $V\subset D-X$ if and
only if Kohn--Rossi cohomology groups $H^{p, q}_{K R}(X)$ are zeros
for $1\le q\le n-2$}
\end{theorem}

When $n=2$, the Plateau problem remains unsolved for many years even
there are no any criterion to judge whether $X$ is the boundary of a
complex manifold. In \cite{Du-Ya}, the first and the third authors
used $CR$ invariant $g^{(1,1)}(X)$ to give the sufficient and
necessary condition for the variety bounded by a Calabi-Yau CR
manifold $X$ being smooth if $H_h^2{(X)}=0$.

\begin{theorem}(\cite{Du-Ya})
\emph{Let $X$ be a strongly pseudoconvex compact Calabi-Yau $CR$
manifold of dimension $3$. Suppose that $X$ is contained in the
boundary of a strongly pseudoconvex bounded domain $D$ in
$\mathbb{C}^N$ with $H_h^2{(X)}=0$. Then $X$ is the boundary of a
complex sub-manifold (up to normalization) $V\subset D-X$ with
boundary regularity if and only if $g^{(1,1)}(X)=0$.}
\end{theorem}

\begin{theorem}(\cite{Du-Ya})\label{h2}
\emph{Let $X$ be a strongly pseudoconvex compact $CR$ manifold of
dimension $3$. Suppose that $X$ is contained in the boundary of a
strongly pseudoconvex bounded domain $D$ in $\mathbb{C}^3$ with
$H_h^2{(X)}=0$. Then $X$ is the boundary of a complex sub-manifold
$V\subset D-X$ if and only if $g^{(1,1)}(X)=0$ .}
\end{theorem}

We will use the new $CR$ invariant $g^{(\Lambda^n 1)}(X)$ to deal
with complex Plateau problem of $X$ in general type.

\begin{theorem}\label{rat}
\emph{Let $X$ be a strongly pseudoconvex compact $CR$ manifold of
dimension $2n-1$, where $n>2$. Suppose that $X$ is contained in the
boundary of a strongly pseudoconvex bounded domain $D$ in
$\mathbb{C}^N$. Then $X$ is the boundary of a variety $V\subset D-X$
with boundary regularity and the number of non-rational
singularities is not great than $g^{(\Lambda^n 1)}(X)$. In
particular,  if $g^{(\Lambda^n 1)}(X)=0$, then $V$ has at most
finite number of rational singularities.}
\end{theorem}
\begin{proof}
It is well known that $X$ is the boundary of a variety $V$ in $D$ with
boundary regularity (\cite{Lu-Ya}, \cite{Ha-La2}). Suppose $\{0_1,
\cdots 0_k\}$ be $k$ non-rational singularities in $V$.  The the
result follows easily from Theorem \ref{newn} and Lemma \ref{gp1}.
\end{proof}

When $X$ is a Calabi--Yau $CR$ manifold of dimension $5$, we give
the following necessary and sufficient condition for the variety
bounded by $X$ being smooth.
\begin{theorem}\label{5d}
\emph{Let $X$ be a strongly pseudoconvex compact Calabi-Yau $CR$
manifold of dimension $5$. Suppose that $X$ is contained in the
boundary of a strongly pseudoconvex bounded domain $D$ in
$\mathbb{C}^N$. Then $X$ is the boundary of a complex sub-manifold
(up to normalization) $V\subset D-X$ with boundary regularity if and
only if $g^{(\Lambda^3 1)}(X)=0$.}
\end{theorem}
\begin{proof}
$(\Rightarrow)$ : Since $V$ is smooth, $g^{(\Lambda^3 1)}(X)=0$ follows from
Lemma \ref{gp1}.

$(\Leftarrow)$ : It is well known that $X$ is the boundary of a
variety $V$ in $D$ with boundary regularity (\cite{Lu-Ya},
\cite{Ha-La2}). The the result follows easily from Theorem
\ref{new3} and Lemma \ref{gp1}.
\end{proof}

\begin{corollary}
\emph{Let $X$ be a strongly pseudoconvex compact Calabi-Yau $CR$
manifold of dimension $5$. Suppose that $X$ is contained in the
boundary of a strongly pseudoconvex bounded domain $D$ in
$\mathbb{C}^4$. Then $X$ is the boundary of a complex sub-manifold
$V\subset D-X$ with boundary regularity if and only if
$g^{(\Lambda^3 1)}(X)=0$.}
\end{corollary}
\begin{proof}
The result follows easily from the fact that isolated hypersurface
singularities are normal and Gorenstein.
\end{proof}

%\end{comment}

\end{document}